\documentclass[11pt, a4paper]{amsart}
 
\usepackage{amsmath,amssymb,amsthm,amscd, amsxtra, amsfonts,graphicx,fancybox}
\usepackage{latexsym}
\usepackage{color}
\usepackage[pdfstartview=FitH]{hyperref}
\usepackage{tikz-cd}
\usepackage{enumerate}
\usepackage{subcaption}
\usepackage{adjustbox}
\usepackage{wasysym}

\newcommand{\cC}{\mathcal{C}}

\newtheorem{theorem}{Theorem}[section]
\newtheorem*{theorem*}{Theorem}
\newtheorem{lemma}[theorem]{Lemma}
\newtheorem{definition}[theorem]{Definition}

\newtheorem{proposition}[theorem]{Proposition}
\newtheorem*{proposition*}{Proposition}
\newtheorem{corollary}[theorem]{Corollary}
\newtheorem*{corollary*}{Corollary}

\newtheorem{remark}[theorem]{Remark}
\def \<{\langle}
\def \>{\rangle}

\def \a{\alpha }

\def \l{\lambda }

\def \o{\omega }
\def \1{\mathbf 1}
\def \I{I_{\geq2}}

\DeclareMathOperator{\Res}{Res}
\DeclareMathOperator{\spn}{span}

\newcommand{\HVir}{\mathcal{H}}
\newcommand{\W}{\mathcal{W}}

\newcommand{\LH}{L^\HVir}
\newcommand{\LW}{L^\W}

\newcommand{\ben}{\begin{enumerate}}
\newcommand{\een}{\end{enumerate}}
\newcommand{\be}{\begin{equation}}
\newcommand{\ee}{\end{equation}}
\newcommand{\wt}{{\rm {wt}}}

\newcommand{\hpr}{h[p,r]}
\newcommand{\vpr}{v_{p,r}}
\newcommand{\Z}{\mathbb Z}
\newcommand{\N}{{\mathbb Z}_{>0} }
\newcommand{\Zp}{{\mathbb Z}_{\ge 0} }
\newcommand{\C}{\mathbb C}

\usepackage{comment}

\newcommand{\nn}{\nonumber \\}
\newcommand{\ec}{\Circle}
\newcommand{\bcr}{\CIRCLE}
 
\usepackage{amssymb}

\title{ The tensor category for  $W(2,2)$--vertex algebra}
    \author[   Adamovi\' c,  Peng,  Radobolja,  Yang ]{ Dra\v zen Adamovi\' c, Mingjie Peng, Gordan Radobolja, Jinwei Yang }
\date{}

\begin{document}

\maketitle

\begin{abstract}
This paper studies the category $\cC$ of grading-restricted $C_1$-cofinite generalized modules for the vertex operator algebra associated to the $W$-algebra $W(2,2)$. We first show that $\cC$ is the same as the category of finite length modules whose simple composition factors are the irreducible highest weight $W(2,2)$--modules $L[r]$ of highest weight $(\frac{1-r}{2}, 0)$ for $r \in \Z_{> 0}$, and hence $\cC$ carries a braided tensor category structure.    Then we prove the fusion rules for the simple objects $L[r]$ are governed by the $\mathfrak{sl}_2$ Clebsch--Gordan rule. In particular, we prove
\[
L[r]\boxtimes L[s]\cong \bigoplus_{i=0}^{\min\{r,s\}-1} L[r+s-1-2i].
\]
Using the fusion rules and a recent result of Etingof--Penneys, we establish the rigidity of $\cC$. We also show the semisimple subcategory generated by the simple objects $L[r]$ is tensor equivalent to the category Rep $\mathfrak{sl}_2$ of finite dimensional $\mathfrak{sl}_2$-modules.
\end{abstract}

\section{Introduction}
The W-algebra $W(2,2)$ was introduced in a 2009 paper \cite{Zhang-Dong} classifying simple vertex operator algebras of the moonshine type generated by a two-dimensional weight-two subspace. The same algebra appears in physics under several names, most notably the Bondi--Metzner--Sachs algebra $\mathfrak{bms}_3$ and the Galilean Conformal Algebra ($\mathrm{GCA}_2$), which is a non-relativistic contraction of two copies of the Virasoro algebra. This contraction procedure has since been used to construct new $W$-algebras, known as Galilean $W$-algebras.

The representation theory of $W(2,2)$ has since been extensively studied; see \cite{AR1,AR2,AR3,JLPZ,JZ,R}. Highest weight modules are naturally parametrized by two parameters. The most interesting cases (known as atypical) occur when both parameters are positive integers, which is exactly when the corresponding Verma module has a subsingular vector. It was shown in \cite{AR1,AR2,AR3} that the simple vertex operator algebra associated to $W(2,2)$, which we denote by $\LW_c$, embeds into the twisted Heisenberg--Virasoro algebra $\mathcal{H}$ at level zero, which yields a free-field realization via a rank-two Heisenberg algebra and certain screening operators. Both of these algebras are neither rational nor $C_2$-cofinite. While the known fusion rules for $\mathcal{H}$ provide some intertwining operators among $\LW_c$-modules, the fusion rules for $\LW_c$ itself have remained unknown.

More generally, it is a natural question to construct and study braided tensor category structures on appropriate representation categories of the $W(2,2)$ vertex operator algebra $\LW_c$. The existence of such a braided tensor category and its rigidity are still a major problem for the vertex operator algebras that are neither rational nor $C_2$-cofinite. The main difficulty is that representation categories of such vertex operator algebras are usually neither finite nor semisimple. Recently, the category of grading-restricted $C_1$-cofinite generalized modules has been studied intensively as it has a few nice properties with respect to the tensor product. In fact, Huang \cite{H} proved the category of grading-restricted $C_1$-cofinite generalized modules is a braided monoidal category, it still remains to prove this category is an abelian category in order to obtain a braided tensor category structure. 

In \cite{C.et.al} and \cite{ALY}, by showing that the category of grading-restricted $C_1$-cofinite generalized modules is the same as the category of finite length modules whose simple composition factors are irreducible $C_1$-cofinite modules, the authors proved that there is a braided tensor category structure on the category of grading-restricted $C_1$-cofinite generalized modules for the Virasoro algebra and for the $\Z_2$-orbifold of the Heisenberg vertex operator algebra, respectively. The latter is also an extension of the Virasoro algebra, named as $W(2,4)$ in the literature. This work motivates us to construct a braided tensor category from the $W(2,2)$ vertex operator algebra. 

We first show the atypical irreducible modules $L[r]$, $r\in\N$, of highest weights $\left(\tfrac{1-r}{2},0\right)$ are the only $C_1$-cofinite irreducible grading-restricted generalized $\LW_c$-modules. Then we prove the category $\cC$ of grading-restricted $C_1$-cofinite generalized modules coincides with the category of finite-length modules whose composition factors are isomorphic to $L[r]$ for $r \in \N$. As a consequence, the category $\cC$ is an abelian category and hence admits a braided tensor category structure. 

We also determine the fusion rules for the simple objects by using Zhu's theory together with the known fusion rules for $\mathcal{H}$-modules. It turns out that the fusion rules are given by the Clebsch--Gordan rule for $\mathfrak{sl}_2$. The fusion rules are very useful and tell much more information about the category $\cC$. Firstly, combined with a recent result of Etingof and Penneys \cite{EP}, we establish the rigidity of $\cC$. Secondly, using the fusion rules and the rigidity, together with Kazhdan-Wenzl's result (\cite[Theorem~$A_\infty$]{KW}), we prove the semisimple subcategory $\cC^s$ consisting of direct sums of the simple objects $L[r]$, $r \in \N$ is tensor equivalent to the category Rep $\mathfrak{sl}_2$ by comparing the intrinsic dimension of the generating object $L[2]$ and that of the standard representation of the quantum group over $\mathfrak{sl}_2$.

The structure of the paper is as follows: In Section 2 we recall the structure of highest weight $W(2,2)$-modules with an emphasis on singular and subsingular vectors in the Verma modules. A partial formula for the subsingular vectors, established in Section 2, underlies the classification of $C_1$-cofinite modules that we carry out in Section 3. Furthermore, we describe the category $\cC$ of $C_1$-cofinite grading-restricted generalized modules and show that it admits the structure of a braided tensor category. In Section 4 we determine the fusion rules for the simple objects and show that $\cC$ is rigid. In Section 5, we prove the semisimple subcategory generated by the simple objects is tensor equivalent to Rep $\mathfrak{sl}_2$.

\section{The algebra $W(2,2)$}
Galilean Conformal Algebra (GCA) or $W(2,2)$ is a Lie algebra
\[\W=\spn_\C\{L(n), W(n),C_{L},C_{W}:n\in\Z\}\]
with a Lie bracket
\begin{align*}
\left[L(n),L(m)\right] & =(n-m)L(n+m)+\delta_{n,-m}\frac{n^{3}-n}{12}C_{L},\\
\left[L(n),W(m)\right] & =(n-m)W(n+m) +\delta_{n,-m}\frac{n^{3}-n}{12}C_{W},\\
\left[W(n),W(m)\right] & =\left[\cdot,C_{L}\right] =\left[\cdot,C_{W}\right] =0.
\end{align*}
There is a natural triangular decomposition
$$\W=\W^-\oplus\W^0\oplus\W^+$$
where
\begin{align*}
\W^{\pm}&=\spn_\C\{L(n),W(n):\pm n>0\}\\
\W^0&=\spn_\C\{L(0),W(0),C_L,C_W\}
\end{align*}
Let also 
\[
\W^{\le0}=\W^-\oplus\W^0,\qquad\W^{\ge0}=\W^0\oplus\W^+.
\]
Consider a one-dimensional $\W^{\ge0}$-module $\C v_{c_L,c_W,h_L,h_W}$, with trivial action of $\W^+$ and 
\begin{align*}
&L(0)v_{c_L,c_W,h_L,h_W}=h_Lv_{c_L,c_W,h_L,h_W},\quad W(0)v_{c_L,c_W,h_L,h_W}=h_Wv_{c_L,c_W,h_L,h_W},\\
&C_Lv_{c_L,c_W,h_L,h_W}=c_Lv_{c_L,c_W,h_L,h_W},\quad C_Wv_{c_L,c_W,h_L,h_W}=c_Wv_{c_L,c_W,h_L,h_W}
\end{align*}
for $h_L,h_W,c_L,c_W\in\C$.
Denote by $$V(c_L,c_W,h_L,h_W)=U(\W)\otimes_{U(\W^+\oplus\W^0)}\C v_{c_L,c_W,h_L,h_W}$$ the Verma module of the highest weight $(h_L,h_W)$ and the central charge $(c_L,c_W)$. Note that the action of $W(0)$ is not semi-simple.

For any $t\in\C^\times$ there is an automorphism $\sigma_t$ of $\W$ such that
\[
\sigma(L(n))=L(n),\quad\sigma(W(n))=tW(n),\quad\sigma(C_L)=C_L,\quad\sigma(C_W)=tC_W
\]
so if $c_Lc_W\ne0$ we may choose $t=\tfrac{c_L}{c_W}$ and normalize $c_W$ to $c_L$. Throughout this paper, we assume that the central charge is non-trivial (i.e.\ $c_L=c_W\neq0$) and denote it by $c$. Denote the Verma module of highest weight $(h_L,h_W)$ and central charge $c$ by $V(c, h_L,h_W)$, and its simple quotient by $L(c, h_L,h_W)$.

We shall parameterize the weights in the following useful way.
For $p,r\in\C$, let $M[p,r]$ denote a highest weight module of highest weight
\begin{align}
&\mathbf{h}[p,r]=(\hpr,h_W[p,r]),\label{param1}
\end{align}
where
\begin{align}
&\hpr=(1-p^2)\frac{c-26}{24}+1-\frac{r+1}{2}p,\label{param2}\\
&h_W[p,r]=\frac{1-p^2}{24}c.\label{param3}
\end{align}
In particular, let $V[p,r]$ denote the Verma module, $\vpr$ its highest weight vector, and $L[p,r]$ the unique irreducible quotient.
Note that we have
\begin{align}
&\mathbf{h}[p,r]=\mathbf{h}[-p,-r-2]\label{sim-p}\\
&\hpr+p=h[p,r-2].
\end{align}
Although $\mathbf{h}[0,r]=\left(\frac{c-2}{24},\frac{c}{24}\right)$ for all $r\in\C$, and weights $\left(\frac{c-2}{24}+\l,\frac{c}{24}\right)$ for $\l\neq0$ are not covered by (\ref{param1}), this restriction does not affect our treatment. Furthermore, we always assume that $\operatorname{Re}{(p)}\geq0$ due to (\ref{sim-p}).

Obviously, the PBW basis of the Verma module consists of all vectors
\begin{align}\label{PBW}
&W(-n)^{w_n}\cdots W(-1)^{w_1}L(-m)^{\ell_m}\cdots L(-1)^{\ell_1}\vpr\quad n,m,w_i,\ell_j\in\Zp.
\end{align}

The following results have been proved in \cite{Zhang-Dong}, \cite{R} and \cite{JLPZ}.
\begin{theorem}
Given a non-zero central charge $c\in\C^\times$, we have:
\begin{description}
	\item[generic case]The Verma module $V(c,h_L,h_W)$ is irreducible unless $(h_L,h_W)=\mathbf{h}[p,r]$ for some $p\in\N$, $r\in\C$. 
    \item[non-generic case]In the case $p\in\N$ and $r\in \C$, $V[p,r]$ has a singular vector
	\begin{equation}\label{sing}
		u_{p}=U_p\vpr,\qquad U_p\in\C[W(-1),\ldots,W(-p)]
	\end{equation}
	of weight $\hpr+p=h[p,r-2]$, such that $U(\W)u_p\cong V[p,r-2]$. Note that $U_p$ can be normalized at $W(-p)$.
\end{description}
\end{theorem}

\begin{theorem}\label{W-struc}
    The structure of the Verma modules in the non-generic case divides into the following two cases:
\begin{description}
	\item[typical case]If $p\in\N$ and $r\notin\N$, then the maximal submodule of $V[p,r]$ is generated by the singular vector $u_p$.
We have the following exact sequence
\[0\longrightarrow V[p,r-2]\longrightarrow V[p,r]\longrightarrow L[p,r]\longrightarrow0.\]
    
\item[atypical case]Assume that $p,r\in\N$. Then the quotient module $V[p,r]/U(\W)u_p$ contains a singular vector $s_{p,r}=S_{p,r}\vpr$ of weight $\hpr+rp=h[p,-r]$, called a subsingular vector in $V[p,r]$, that is of the form
\begin{equation}\label{subsing}
S_{p,r}=L(-p)^r+\sum_{i = 1}^{r}c_ig_iL(-p)^{r-i}
\end{equation}
where $c_i \in \C$ depends on the central charge $c$, $g_i \in U(\W^-)$ that contains at least one factor $W(-k)$ for $k \in \Z_{>0}$, but does not involve $W(-p), L(-p)$. 
    
Moreover, the maximal submodule of $V[p,r]$ is
    \begin{equation}\label{max-sub-gen}
    	\overline V[p,r]=U(\W)s_{p,r}=U\left(\W^{\le0}\right)\spn_\C\{u_{p},s_{p,r}\}
    \end{equation}
    so we have the following exact sequences
\begin{eqnarray}
    &0\longrightarrow\overline V[p,r]\longrightarrow V[p,r]\longrightarrow L[p,r]\longrightarrow0&\nn
    &0\longrightarrow V[p,r-2]\longrightarrow\overline V[p,r]\longrightarrow L[p,-r]\longrightarrow0&\label{max-sub}.
\end{eqnarray}
 In particular, the singular vector $u_p \in U(\W)s_{p,r}$, and the quotient module of $V[p,r]$ by $V[p,r-2]$ is a non-split extension of two irreducible modules:
\[
0\longrightarrow L[p,-r]\longrightarrow V[p,r]/V[p,r-2]\longrightarrow L[p,r]\longrightarrow 0.
\]   
\end{description}
The PBW basis of the irreducible module $L[p,r]$ consists of all vectors (\ref{PBW}) such that
\begin{description}
	\item[typical case]$w_p=0$ if $p\in\N$ and $r\notin\N$;
	\item[atypical case]$w_p=0$ and $\ell_p<r$ if $p,r\in\N$.
\end{description}
\end{theorem}

\begin{remark}\label{cof}
Although the exact formula for the subsingular vector (\ref{subsing}) is unclear, it is known that in PBW basis each term of $s_{p,r}-L(-p)^r\vpr$ contains at least one factor $W(-k)$ for $k \in \Z_{>0}$. From here, and the commutator relations in $\W$ we can see that, no matter the choice of basis, $L(-1)^{rp}\vpr$ cannot occur in $s_{p,r}$ if $p>1$. This will be crucial in the classification of the $C_1$-cofinite highest weight modules.
\end{remark}

We illustrate the structure of the atypical Verma module $V[p,r]$ by a diagram in Figure \ref{fig}. Filled circles denote singular, and empty ones subsingular vectors. Arrows point to the submodules. The top vector has conformal weight $h[p,r]$, and each level below it increases weight by $p$. Thus, conformal weights of $u_{p,r}^{(k)}$ and $s_{p,r}^{(k)}$ are $h[p,r]+kp=h[p,r-2k]$.

There is an infinite list of singular vectors $u_{p,r}^{(i)}=U_p^i\vpr$, $i\in\N$, each generating a submodule isomorphic to $V[p,r-2i]$ and a finite list of subsingular vectors $s_{p,r}^{(k)}=S_{p,r-2k+2}v_{p,r-2k+2}\in V[p,r-2k+2]$, $k=1,\ldots,\lceil\tfrac{r+1}{2}\rceil$.

{\small
\begin{figure}[h]
\adjustbox{scale=0.8}{\begin{tikzcd}
  	&&&\phantom{\vpr\quad}\bcr\quad\vpr\arrow[to=8-1]\\
	&&&\phantom{u^{(1)}_{p,r}\quad}\bcr\quad u^{(1)}_{p,r}\arrow[to=7-2]\\
	&&&\phantom{u^{(2)}_{p,r}\quad}\bcr\quad u^{(2)}_{p,r}\arrow[to=6-3]\\
	&&&\phantom{u^{(3)}_{p,r}\quad}\bcr\quad u^{(3)}_{p,r}\\
	&&&\vdots\arrow[d]\\
	&\qquad\qquad\qquad&\phantom{s^{(3)}_{p,r}}\quad\ec\quad s^{(3)}_{p,r}\arrow[to=4-4]&\phantom{u^{(r-2)}_{p,r}\quad}\bcr\arrow[d]\quad u^{(r-2)}_{p,r}\\
	\qquad&\phantom{s^{(2)}_{p,r}}\quad\ec\quad s^{(2)}_{p,r}\arrow[to=3-4]&&\phantom{u^{(r-1)}_{p,r}\quad}\bcr\arrow[d]\quad u^{(r-1)}_{p,r}\\
	\phantom{s^{(1)}_{p,r}}\quad\ec\quad s^{(1)}_{p,r}\arrow[to=2-4]&&&\phantom{u^{(r)}_{p,r}\quad}\bcr\arrow[d]\arrow[d]\quad u^{(r)}_{p,r}\\
	&&&\vdots
  \end{tikzcd}}
\caption{Atypical Verma module $V[p,r]$, $p,r\in\Zp$}\label{fig}
\end{figure}
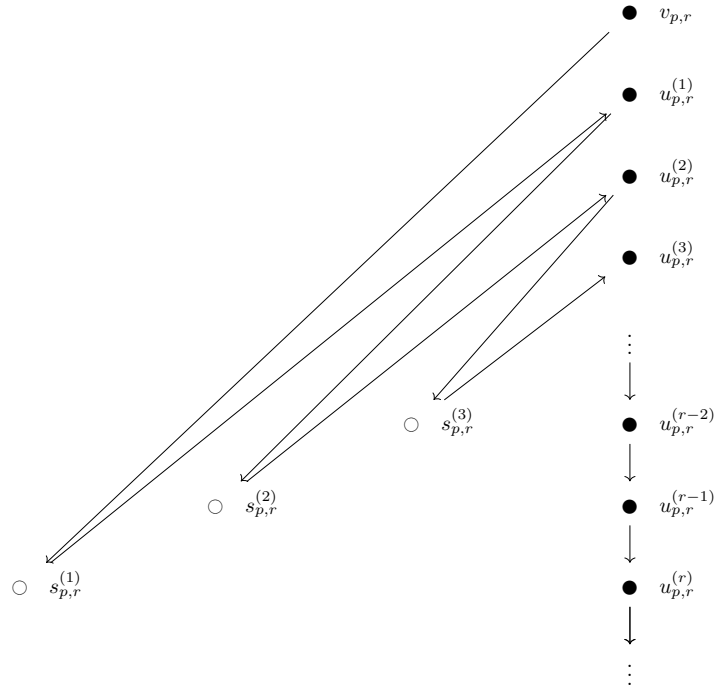}

\section{$C_1$-cofinite modules}

For each $c\in\C^\times$, the vacuum module $L[1,1]=L(c,0,0)$ carries the structure of a simple vertex operator algebra, strongly generated by the fields
\[\o(z)=\sum_{n\in\Z}L(n)z^{-n-2},\quad W(z)=\sum_{n\in\Z}W(n)z^{-n-2}.\]
We denote this vertex operator algebra by $\LW_c$.

\begin{definition}
A module $M$ over a vertex operator algebra $V$ is $C_1$-cofinite if $$C_1(M)=\spn_\C\big\{v_{-1}x\mid x\in M, v\in \coprod_{n \in \Z_{\geq 1}}V_{(n)}\big\}$$ has finite codimension in $M$.
\end{definition}

Denote by $\I$ the two-sided ideal in the universal enveloping algebra $U(\W^-)$, generated by $\{L(-n),W(-n):n\geq2\}$. Since $[W(-1),L(-1)]=0$, we have
\ben
	\item $U(\W^-)/\I\cong\C[x,y]$ under the map $[L(-1)] \mapsto x$ and $[W(-1)] \mapsto y$,\label{I2-quot}
	\item $C_1(M) = \I M$ for any $\LW_c$-module $M$.
\een

\begin{proposition} For each $r \in {\Z}_{>0}$, there is a subsingular vector $s_{1,r}\in V[1,r]$ such that
\[s_{1,r} \equiv L(-1) ^r v_{1,r} \quad \big({\rm mod}\; \; C_1 (V[1,r])\big).\]
Moreover, the simple quotient $L[1,r]$ is $C_1$-cofinite.
\end{proposition}

\begin{proof}
By Theorem \ref{W-struc} (cf.\ \cite{R} and \cite{JLPZ} for more details) there is a subsingular vector (\ref{subsing}) of the type
\begin{equation}\label{subsing p=1}
s_{1,r}=\left(L(-1)^r+\sum_{k=0}^{r-1} w_{r-k}L(-1)^k\right)v_{1,r}
\end{equation}
where $w_{r-k}\in U(\W^-)$ does not involve $W(-1)$ and $L(-1)$. Therefore, we have 
\[s_{1,r} - L(-1) ^r v_{1,r}  \in   C_1(V[1,r]).\]
In the simple quotient $L[1,r]$, we have $W(-1) v_{1,r}=0$ and $s_{1,r} =0$ which implies that  $L(-1) ^r v_{1,r} \in 
 C_1(L[1,r])$. From this we easily conclude that $L[1,r]$ is $C_1$-cofinite.
\end{proof}

We will show that $L[1,r]=L\left(\tfrac{1-r}{2},0\right)$, $r\in\N$ are the only $C_1$-cofinite modules among the highest weight $\LW_c$-modules.

\begin{definition}
Let $\cC$ denote the category of grading-restricted generalized $\LW_c$-modules that are $C_1$-cofinite. For $r\in\N$, define \[L[r]:=L[1, r].\]
\end{definition}

\begin{proposition}\label{prop:classification}
Every highest weight module in $\cC$ is irreducible and isomorphic to $L[r]$ for some $r\ge 1$.
\end{proposition}
\begin{proof}
We have already shown that modules $L[r]$ are $C_1$--cofinite. Let $M[p,r]$ be an arbitrary highest weight $\W$-module. By definition, $C_1(M[p,r])$ contains
\[
\spn_\C\{L(-n-1)x,W(-n-1)x\mid n\in\N, x\in M[p,r]\}
\]
Hence for $C_1$-cofiniteness of $\W$-modules we need to show that
\[
\spn_\C\{W(-1)^wL(-1)^\ell\vpr+C_1(M[p,r])\mid w,\ell\in\Zp\}
\]
is finite-dimensional.

But from the structure of highest weight $\W$-modules in Theorem \ref{W-struc} (see also Remark \ref{cof}), it follows that $\left\lbrace L(-1)^\ell\vpr+C_1(M[p,r]):\ell\in\Zp\right\rbrace$ is finite-dimensional if and only if $p=1$ and $M[p,r]=L[r]$ for some $r\in\N$.
\end{proof}

In the following we show that category $\cC$ is exactly the category of finite length $\LW_c$-modules whose composition factors are $L[r]$.

\begin{lemma} \label{intersection}
Let
\[
  0\longrightarrow N\longrightarrow M\longrightarrow L[r]\longrightarrow 0
\]
be an exact sequence of $\LW_c$-modules, with $r\in\N$. Let $w\in M$ be a lift of the highest-weight vector
$v_{1,r}\in L[r]$. Assume $b\in \I$ is such that $b v_{1,r}=0$. Then $bw\in C_1(N)$.
\end{lemma}
\begin{proof}
Applying (\ref{I2-quot}) and (\ref{subsing p=1}) we write
\begin{align*}
&x=L(-1),\qquad y=W(-1),\\
&s_{1,r}=S_r v_{1,r},\qquad S_r=x^r+R_r\in U(\W^-),
\end{align*}
where each PBW monomial of $R_r$ contains $W(-n)$ for some $n>1$. Thus $R_r\in\I$.

By Theorem \ref{W-struc} (\ref{max-sub-gen}), our assumption $b v_{1,r}=0$ leads to $bw\in N$ and
\[
	b=By+B'S_r,\qquad\text{for some}\qquad B,B'\in U(\W^-)
\]
Choose representatives of the images of $B$ and $B'$ in
$U(\W^-)/\I\cong\C[x,y]$:
\[
  B=B_0+B_1,\qquad B'=B'_0+B'_1,
\]
with
\[
  B_0,\,B'_0\in\C[x,y],
  \qquad
  B_1,\,B'_1\in\I.
\]
Because $b, R_r \in I_{\geq 2}$, we apply the canonical projection $U\longrightarrow U/\I\cong\C[x,y]$ to
\[
	b=B_0y+B'_0x^r+B'_0R_r+B_1y+B'_1S_r
\]
and get 
\[
B_0y+B'_0x^r=0\qquad\text{in }\C[x,y].
\]
Since $x^r$ and $y$ are relatively prime, there exists
$h\in\C[x,y]$ such that
\[
B_0 = x^r h,\qquad B'_0 = -yh.
\]
When lifting to $M$ we have $yw,\,S_rw \in N$, so 
\[
  B_1 y w,\ B'_1 S_r w\in C_1(N).
\]
Thus, modulo $C_1(N)$, we have
\begin{align*}
	bw &\equiv (B_0 y + B'_0 S_r) w \\
		&\equiv (x^r h y - yh (x^r+R_r))w\\
		&\equiv -y h R_r w.
\end{align*}

Since $y=W(-1)$ commutes with $x$ and with all $W(-m)$, and since $yw\in N$ we have $yR_rw\in C_1(N)$, and since $C_1(N)$ is stable under multiplication by $h$, we also have $yhR_rw\in C_1(N)$. Thus $bw\in C_1(N)$ as claimed.
\end{proof}

\begin{proposition}\label{prop:extension}
Let
\[
  0\longrightarrow N\longrightarrow M\longrightarrow L[r]\longrightarrow 0
\]
be exact, with $r\in\N$. If $M$ is $C_1$-cofinite and $N$ is finitely generated, then $N$ is $C_1$-cofinite.
\end{proposition}

\begin{proof}
Let $w\in M$ lift the highest-weight vector of $L[r]$. The decompos
\[
  M=N+U(\W^-)w,\qquad\text{implies}\qquad
  C_1(M)=C_1(N)+\I w.
\]

The operators $x=L(-1)$ and $y=W(-1)$ preserve $C_1(M)$ and raise conformal weight by $1$. Since $M/C_1(M)$ is finite-dimensional and generalized $L(0)$-graded, $x$ and $y$ act nilpotently on $M/C_1(M)$. Thus for $u\in N$, $A\in\{L,W\}$ and $K\gg0$, we have
\[
  A(-1)^Ku\in C_1(M).
\]
Write
\[
  A(-1)^Ku=a+bw,
  \qquad
  a\in C_1(N),\quad b\in I_{\ge2}.
\]
Since $A(-1)^Ku\in N$, we have $bw\in N$, i.e., $b v_{1,r}=0$ in $L[r]$. By Lemma \ref{intersection} we get $bw\in C_1(N)$.
Thus $
  A(-1)^Ku\in C_1(N)$.

Choose finite generators $u_1,\ldots,u_m$ of $N$. Modulo $C_1(N)$, every vector of $N$ is represented as a linear combination of
\[
  W(-1)^aL(-1)^b u_j,\qquad 1\le j\le m.
\]
For each $j$, sufficiently high powers of $L(-1)$ and $W(-1)$ send $u_j$ into $C_1(N)$. Since $C_1(N)$ is stable under $L(-1)$ and $W(-1)$, only finitely many such classes survive. Hence $\dim N/C_1(N)<\infty$ proving the claim.
\end{proof}

\begin{proposition}
Every grading-restricted $C_1$-cofinite $\LW_c$-module $M$ admits a finite filtration
\[
  0=M_0\subset M_1\subset\cdots\subset M_n=M
\]
whose successive quotients are highest-weight $\LW_c$-modules.
\end{proposition}

\begin{proof}
The proof is the same as the Virasoro case (Proposition 3.1.2. in \cite{C.et.al}), so we omit it here.
\end{proof}

\begin{theorem}\label{C1fl}
Let $M$ be a grading-restricted $C_1$-cofinite $\LW_c$-module. Then $M$ has a finite filtration whose successive irreducible factors are of the form $L[r]$, for some $r\in\N$.
Moreover, every submodule occurring in such a filtration is
$C_1$-cofinite.
\end{theorem}

\begin{proof}
By the previous proposition, $M$ has a finite highest-weight filtration.
Starting with $M$ and descending through the filtration, each quotient is a quotient of a $C_1$-cofinite module and is therefore $C_1$-cofinite.
By Proposition \ref{prop:classification} these factors are precisely $L[r]$, $r\in\N$. The Proposition \ref{prop:extension} then shows, by descending induction, that every filtration submodule is $C_1$-cofinite.
\end{proof}
%

\begin{theorem}\label{thm:Huang}
The category $\cC$ admits a braided tensor category structure. In particular, the tensor product in $\cC$ is associative.
\end{theorem}
\begin{proof}
    By \cite{H}, the category $\cC$ is a braided monoidal category. Since it is the same as the category of finite length modules whose composition factors are $L[r]$ for $r \in \Z_{>0}$, it is also an abelian category. Therefore $\cC$ is a braided tensor category.
\end{proof}

\section{The fusion rules in the category $\mathcal C$}

In this section we calculate some fusion rules for $C_1$--cofinite modules.  We  apply Frenkel-Zhu's formula for fusion rules \cite{FZ}, with some refinement made by H. Li in \cite{Li}. As a consequence, we determine the tensor product $L[r] \boxtimes  L[s]$ in the category $\cC$ and prove rigidity using the Etingof--Penneys  criterion from \cite{EP}.

\subsection{Zhu algebras and Frenkel-Zhu bimodules}
 
\begin{definition} \cite{FZ}
Let $V$ be a vertex operator algebra and let $M$ be a $V$-module. Set
\[
O(M)=\operatorname{span}\left\{\Res_z\, Y(a,z)m\frac{(1+z)^{\wt(a)}}{z^2}\;\middle|\; a\in V,\; m\in M\right\}.
\]
The quotient
\[
A(M)=M/O(M)
\]
is called the \emph{Frenkel-Zhu bimodule} of $M$. It carries a natural $A(V)$-bimodule structure induced by the operations: 
\begin{align*}
a*m &= \Res_z\, Y(a,z)m\frac{(1+z)^{\wt(a)}}{z},\\
m*a &= \Res_z\, Y(a,z)m\frac{(1+z)^{\wt(a)-1}}{z}
\end{align*}
for homogeneous $a\in V$ and $m\in M$.
\end{definition}

We are focused on proving the existence of intertwining operators. H.\ Li proved the following refinement of Frenkel-Zhu's theorem.
 
\begin{theorem} \cite{Li} \label{li}
\item[(1)]Assume that $M_1, M_2, M_3 $  are  $V$-modules and $M_3$ is irreducible. Then there is injective homomorphism
 $$ \pi :  I\binom{M_3}{M_1\ M_2} \hookrightarrow  \operatorname{Hom}_{A(V)} \left(A(M_1) \otimes M_2(0), M_3(0) \right). $$ 

\item[(2)]If $M_2 $ and $M_3 '$ are generalized Verma $V$-modules, then $\pi$ is an isomorphism.

\end{theorem}

Let $V=\LW_c$ be the simple vertex operator algebra of $W(2,2)$. One has
\[
A(V)\cong\C[x,y],
\]
where $x=[\o]$ and $y=[W]$.

For any $V$-module $M$ and any $m\in M$, the left and right actions of $x$ and $y$ are given by
\begin{align}
[\o]*[m] &= [(L({-2})+2L({-1})+L(0) )m], \label{eq:left-omega}\\
[m]*[\o] &= [(L({-2})+L({-1}))m], \label{eq:right-omega}\\
[W]*[m] &= [(W({-2})+2W({-1})+W(0) )m], \label{eq:left-w}\\
[m]*[W] &= [(W(-2)+W({-1}))m]. \label{eq:right-w}
\end{align}

Let $M$ be a $\Zp$--graded $V$-module, generated by the top space $M_{top}$.
We equip 
\[\C[x_1, x_2, y_1, y_2] \otimes  M_{top}\]
with the structure of $\C[x,y]$-module such that the left action of $\C[x,y]$ is through $x_1, y_1$ and the right action through $x_2, y_2$. 

\begin{proposition}
We have a surjective homomorphism  
 $$ \Phi: \C[x_1, x_2, y_1, y_2] \otimes  M_{top} \rightarrow A(M), $$
uniquely determined by 
\begin{align*}
\Phi(x_1\otimes[v])&=[((L(-2)+2L(-1)+L(0))v)],\\
\Phi(x_2\otimes[v])&=[((L(-2)+L(-1))v)],\\
\Phi(y_1\otimes[v])&=[(W({-2})+2W({-1})+W(0))v],
\\
\Phi(y_2\otimes[v])&=[(W({-2})+W({-1}))v],
\end{align*}
where $v \in M_{top}$.
If $M$ is a Verma module, then $\Phi$ is an isomorphism. 
\end{proposition}
 
Consider now the irreducible module
$
M_1=L[1,2]
$
with highest weight vector $v$. It is realized as a quotient of $V[1,2]$ by the maximal submodule $\overline V[1,2]$ generated by singular and subsingular vectors
\begin{equation}\label{eq:singrels}
u_1=W(-1)v 
\qquad
s_{1,2}=\left(L(-1)^2+\frac{6}{c}W(-2)\right)v 
\end{equation}

\begin{proposition}\label{prop:bimodule-presentation}
There is an isomorphism of $A(V)$--bimodules:
\[
\frac{\mathbb C[x_1,x_2,y_1,y_2]}{\left(y_1-y_2,\; (x_1-x_2)^2-\frac14+\frac{6}{c}y_2\right)}
\cong  A(M_1),
\]
sending $1$ to $[v]$.
\end{proposition}

\begin{proof}

We have that the $A(\overline V[1,2])$ is an $A(V)$ sub-bimodule of $A(V[1,2])$
generated by images of generators of $\overline V[1,2]$:
\begin{align*}
 [W(-1) v] &\mapsto (y_1-y_2);\\
 \left[\left(L(-1)^2+\frac{6}{c}W(-2)\right)v\right] &\mapsto  (x_1 - x_2) ^2 - \frac{1}{4} + \frac{6}{c} y_2. 
\end{align*}
The assertion now follows from \[A(M_1) =\frac{A(V[1,2])}{A(\overline V[1,2])}.\]
\end{proof}

\begin{theorem}\label{thm:FZ-general}
Let
$
M_2=L(c,h_2,w_2),
M_3=L(c,h_3,w_3)
$
be irreducible highest weight modules with one-dimensional top levels
$M_2(0)$ and $M_3(0)$.
Then
\[
\dim I\binom{M_3}{M_1\ M_2}\le 1,
\]
and this dimension can be nonzero only if
\[
w_3=w_2,
\qquad
(h_3-h_2)^2=\frac14-\frac{6}{c}w_2.
\]

In particular, if $M_2=M[p,r]$ and $M_3=L[p',r']$, we have
\[
\dim I\binom{L[p',r']}{L[1,2]\ \ M[p,r]}\le 1,
\]
and this dimension can be nonzero only if $p'=p$ and for $p\neq 0$ if
\[
r'=r\pm 1.
\]
\end{theorem}

\begin{proof}
Using  results from \cite{Li, FZ}  we get an injective map
\[
I\binom{M_3}{M_1\ M_2}
\hookrightarrow
\operatorname{Hom}_{A(V)}\bigl(A(M_1)\otimes_{A(V)}M_2(0),\, M_3(0)\bigr).
\]
Now $M_2(0)\cong \mathbb C_{(h_2,w_2)}$ so tensoring over the right
$A(V)$--action amounts to substitutions $x_2=h_2$ and $y_2=w_2$.
Thus Proposition \ref{prop:bimodule-presentation} yields a surjection
\[
\frac{\mathbb C[x_1,y_1]}
{\left(y_1-w_2,\; (x_1-h_2)^2-\frac14+\frac{6}{c}w_2\right)}
\twoheadrightarrow
A(M_1)\otimes_{A(V)}M_2(0).
\]
Hence $A(M_1)\otimes_{A(V)}M_2(0)$ is a cyclic left $A(V)$-module.
A nonzero $A(V)$-module map to
$M_3(0)\cong \mathbb C_{(h_3,w_3)}$ can exist only if the defining
relations vanish at $(x_1,y_1)=(h_3,w_3)$. Hence, necessarily
\[
w_3=w_2,
\qquad
(h_3-h_2)^2=\frac14-\frac{6}{c}w_2.
\]
We get that the Hom-space is at most one-dimensional. Therefore
\[
\dim I\binom{M_3}{M_1\ M_2}\le 1.
\]

Consider now the case
$M_2=M[p,r]$, $M_3=L[p,r']$. Recall parametrization of weights (\ref{param2})--(\ref{param3}).
The condition $w_3=w_2$ then yields $p'=p$ (due to (\ref{sim-p}) we can drop the sign change). Since
\[
h[p,r']-h[p,r]= -\frac{p}{2}(r'-r),
\]
the second condition then becomes
\[
\left(-\frac{p}{2}(r'-r)\right)^2
=\frac14-\frac6c h_W[p,r]
=\frac14-\frac6c\cdot\frac{1-p^2}{24}c
=\frac{p^2}{4}.
\]
Hence
\[
\frac{p^2}{4}(r'-r)^2=\frac{p^2}{4}.
\]
If $p\neq 0$, then $(r'-r)^2=1$, so $r'=r\pm 1$.
The proof follows.
\end{proof}

\begin{remark} 
From the proof of Theorem \ref{thm:FZ-general} it follows that we can have logarithmic intertwining operators of type
 $\binom{M_3}{M_1\ M_2}$ where $M_3$ is a logarithmic module only if $w_2 = \frac{c}{24}$, i.e., when $p=0$. But when $w_2 \ne \frac{c}{24}$ an intertwining operator can exist only for non-logarithmic modules $M_3$.
 \end{remark}

 \subsection{The existence of intertwining  operators}
Now we should prove the existence of some intertwining operators predicted by Theorem \ref{thm:FZ-general}.

 \begin{proposition} \label{prop:construction} For each $ r \in {\Bbb Z}_{>1}$,   
 there exist  non-trivial  intertwining operators of types
\[\binom{L[r \pm 1]}{L[2]\ \ L[r] }.\]
 \end{proposition}
 \begin{proof}
 We shall apply Theorem \ref{li} (2) in the cases
 \[M_1 = L[1,2],  \ M_2 = V[1,r],   \  M_3 = V[1, r\pm 1].\]
Then $A(M_1) \otimes_{A(V)} M_2 (0) \cong \C_{h[1, r-1], h_W[1, r-1]}  \oplus \C_{h[1, r+1], h_W[1, r+1]}$, and we do have non-trivial intertwining operators
\[\binom{M_3 '}{M_1  \ M_2}.\]
We denote these intertwining operators with $\mathcal {Y}_{ r \pm 1} (\cdot, z)$.
 
Recall that the maximal submodule  $\overline V[1,r]$ of $V[1,r]$ is generated by singular vector  $W(-1)v$ and subsingular vector $s_{1,r}$. So we have exact sequence 
\[0 \rightarrow\overline V[1,r] \rightarrow V[1,r] \rightarrow L[1,r] \rightarrow 0.\]
Applying the dual functor for $r' = r \pm 1$,  we get exact sequence
\[0 \rightarrow L[1,  r'] \rightarrow V[1,r'] ' \rightarrow\overline V[1,r'] ' \rightarrow 0.\]

We shall now prove that
\[(*) \quad  \mathcal Y_{ r \pm 1} (u, z)\overline  V[1, r] \in \left(V[1, r'] '  \setminus  V[1, r'] ' _{top}\right)\{z\}.\]
(Here we only need that  $L[1,2] \boxtimes\overline V[1,r]$ cannot contain any module whose top component is isomorphic to $\C_ {h[1, r'], h_W[1, r']}$.)

Denote the highest weight vector of $V[1,r]$ by $v[r]$, highest weight vector of $L[1,2]$ by $v$, and by $v'[r']$ the vector in $V[1, r'] ' $  which generates $L[1, r']$.

By construction above there is $\a_0 \in \C$ such that 
\[v_{\a_0} v[r] =  v'[r'],\]
and $v_{\a} v[r] =0$ for $\a >\a_0$.

Consider now the restriction of  $ Y_{ r \pm 1} (\cdot , z)  $ to the submodule generated by $W(-1) v[r]$. Assume that $v'[r']$ is in the image of this restriction. This implies that there exist $\beta_0, \nu  \in \C$ such that
\[v_{\beta_0} W(-1) v[r] =\nu v'[r'].\]
By a weight argument, we conclude that $\beta_0 = \a_0+1$.

Next we notice that $W(-1)$ acts as a derivation which implies
\[  0= W(-1) (v_{\beta_0} v[r]) = (\underbrace{W(-1) v}_{=0} ) _{\beta_0} v[r] + v_{\beta_0} (W(-1) v[r]), \]
hence $v_{\beta_0} W(-1) v[r] = \nu v'[r'] =0$ so $\nu =0$. Therefore $v'[r']$ is not contained in  $L[1,2] \cdot \left(\LW_c \cdot (W(-1) v[r])\right)$.

Next we consider subsingular vector $s_{1,r}$. Recall that it generates a sub-quotient isomorphic to $ L[1, -r] $.

Assume that
$v'[r']$ is contained in  $L[1,2] \cdot (\LW_c . s_{1,r})$. This would imply that we have a non-zero intertwining operator in
$$   I \binom{ L[r'] }{L[2] \ L[-r] }.    $$
But applying Theorem \ref{thm:FZ-general}(2) we conclude that 
 $I\binom{L[r_3]}{L[2]\ \ L[-r]}$ can be non-zero only if $r_ 3 \in\{- r-1, -r+1\}$.
 Since $$r' = r \pm 1 \notin  \{- r-1, -r+1\}, $$ we get  a contradiction. Therefore (*) holds.

This shows that we do have well-defined, non-zero quotient intertwining operators  $\overline { \mathcal Y}_{\pm}$ of types  $$\binom{Z[r'] }{L[1,2]\ \ L[1,r] }, $$
where $Z[r']$ is a quotient of $V[1,r']' $ (possibly equal to $V[1,r']' $ ) which contains submodule isomorphic to $L[1,r']$.

Next we recall that $L[1,2]$ and $L[1,r]$ are $C_1$-cofinite. By applying main theorem of Miyamoto from \cite{Miy} we conclude that image of $\overline { \mathcal Y}_{\pm}$  should be also $C_1$-cofinite. Since $L[1,r'] $ is the $C_1$-cofinite part of module $V[1,r']'$ (see (\ref{max-sub})) we conclude that we do have non-trivial intertwining operators of types
\[\binom{L[r \pm 1]}{L[2]\ \ L[r] }.    \qedhere\]
\end{proof} 
 
\begin{remark}
The existence of the intertwining operators of the type
\[\binom{L[1,r \pm 1]}{L[1,2]\ \ L[1,r] }\] can also be proved by using intertwining operators for the Heisenberg-Virasoro vertex algebra $\LH_{c,c_{L,I}}$ constructed in \cite{AR1} and the embedding of the $\LW_c \hookrightarrow \LH_{c,c_{L,I}}$ from \cite{AR2}.
 \end{remark}

\subsection{The fusion rules} 
  
Recall that  $\cC$ denotes the category of grading-restricted generalized $V$-modules that are $C_1$-cofinite.

From Theorem  \ref{thm:FZ-general}  and Proposition \ref{prop:construction} we get:

\begin{proposition}\label{thm:fusion-with-fundamental}
For every $r\in\N$ we have the following fusion rules in the category $\cC$:
\begin{align}\label{fusion-2xr}
&L[2]\boxtimes L[1]\cong L[2],\\
&L[2]\boxtimes L[r]\cong L[r-1]\oplus L[r+1],\qquad r>1.
\end{align}
\end{proposition}
\begin{proof}
The existence of (\ref{fusion-2xr}) is clear since $L[1]=\LW_c$. Also note that the non-trivial operator $\binom{L[0]}{L[2]\ \ L[1]}$ cannot exist since $L[0]$ is not $C_1$-cofinite. The rest follows from Proposition \ref{prop:construction} and Theorem \ref{thm:FZ-general}.
\end{proof}

Now we can state the main result about fusion among semi-simple $C_1$-cofinite highest weight modules.

\begin{theorem}\label{thm:main-fusion-ring}
We have the following fusion rules in $\cC$: \[
L[r]\boxtimes L[s]
\cong
\bigoplus_{j=0}^{\min(r-1,s-1)}L[r+s-1-2j].
\]
\end{theorem}

\begin{proof}
The fusion rule in Proposition \ref{thm:fusion-with-fundamental} 
shows that $L[2]$ generates the semi-simple tensor subcategory $\cC^{s}$ of $\cC$ 
which contains all $L[r], \ r \in \N$.
The associativity of tensor product now easily implies the fusion rules
\begin{align*}
L[r]\boxtimes L[s]&=L[2]\boxtimes\left(L[r-1]\boxtimes L[s]\right)-L[r-2]\boxtimes L[s]\\
&
\cong\bigoplus_{j=0}^{\min(r-1,s-1)}L[r+s-1-2j]
\end{align*}
by induction.
\end{proof}

 \subsection{Rigidity}
 \begin{theorem}
 The tensor category $\cC$ is rigid.
 \end{theorem}
 \begin{proof}
 By applying  \cite[Theorem~1.1]{EP}, in order 
to prove the rigidity of $\mathcal{C}$ it is enough to verify that, for every $r \in \N$, the simple object $L[r]$ is non-negligible and of moderate growth.

Recall that an indecomposable object $X$ is said to be \emph{non-negligible} if there exists an object $Y$ such that ${\bf 1}$ is a direct summand of $Y \otimes X$. It is said to be of \emph{moderate growth} if there exists $n \in \N$ such that
\[
\dim \operatorname{End}(X^{\otimes n}) < n!.
\]

Now from the fusion rule
\[
L[r] \boxtimes L[r]
=
L[2r-1] \oplus L[2r-3] \oplus \cdots \oplus L[1],
\]
it follows immediately that the tensor unit ${\bf 1}=L[1]$ occurs as a direct summand of $L[r] \boxtimes L[r]$. Hence $L[r]$ is non-negligible for every $r \in \mathbb{Z}_{>0}$.

On the other hand, the fusion rules among the objects $L[r]$ agree with those in the category of finite-dimensional $\mathfrak{sl}_2$-modules. Since the latter category is known to be of moderate growth (see, for example, \cite{EP}), it follows that each $L[r]$ is of moderate growth.

Therefore, all simple objects of $\mathcal{C}$ are non-negligible and of moderate growth. By \cite[Theorem~4.4.1]{CMY2}, which asserts that a category of finite-length modules is rigid as long as all of its simple objects are rigid, we conclude that $\mathcal{C}$ is rigid.
\end{proof}

\begin{remark}
    Rigidity of the generating object $L[2]$ can also be proved by deriving differential equations satisfied by the matrix elements of products and iterates of intertwining operators (see for example the proof of \cite[Theorem~4.1]{MY1}, cf. Proposition \ref{dim} below).
\end{remark}

\section{Correspondence to Rep $\mathfrak{sl}_2$}
Now we have shown that the semisimple rigid subcategory $\cC^s$ consisting of direct sums of the simple objects $L[r]$ for $r \in \Z_{>0}$ obeys the $\mathfrak{sl}_2$-fusion rule. In this section, we will show it is indeed tensor equivalent to the category of finite dimensional $\mathfrak{sl}_2$-modules via quantum group tensor categories.

For $q \in \mathbb{C}^{\times}$, let $\cC(q, \mathfrak{sl}_2)$ be the category of finite-dimensional weight modules for Lusztig's modified form of the Drinfeld-Jimbo quantum group $U_q(\mathfrak{sl}_2)$ specialized at $q$. When $q = \pm 1$ or $q$ is not a root of unity, $\cC(q, \mathfrak{sl}_2)$ is semisimple with the same fusion ring as Rep $\mathfrak{sl}_2$ \cite{L}. Moreover, the tensor category $\cC(q, \mathfrak{sl}_2)$ is generated by the $2$-dimensional standard representation $L_q(1)$, in the sense that every simple object of $\cC(q, \mathfrak{sl}_2)$ is a subquotient of a tensor power of $L_q(1)$.

Conversely, Kazhdan and Wenzl proved the following reconstruction theorem (\cite[Theorem~$A_{\infty}$]{KW}): If $\cC$ is a rigid semisimple tensor category with a simple generating object $X$, and the fusion ring of $\cC$ agrees with that of Rep $\mathfrak{sl}_2$ under the identification of $X$ with the vector representation of $\mathfrak{sl}_2$, then there exists $q_{\cC}\in\C$, either $q_{\cC} = \pm 1$ or $q_{\cC}$ is not a root of unity, and with $q_\cC^2$ uniquely determined up to inversion, such that $\cC$ is tensor equivalent to $\cC(q_{\cC}, \mathfrak{sl}_2)^\tau$ under a functor that sends $X$ to $L_{q_{\cC}}(1)$. Here $\tau$ denotes modification of the associativity isomorphisms in $\cC(q_{\cC}, \mathfrak{sl}_2)$ by a $3$-cocycle on $\Z/2\Z$. There is only one non-trivial $3$-cocycle $\tau$ on $\Z/2\Z$ up to coboundaries: it changes the usual associativity isomorphism $L_{q_\cC}(1)\otimes (L_{q_\cC}(1)\otimes L_{q_\cC}(1)) \rightarrow (L_{q_\cC}(1)\otimes L_{q_\cC}(1))\otimes L_{q_\cC}(1)$ in $\cC(q_{\cC}, \mathfrak{sl}_2)$ by a sign. Moreover, by \cite[Theorem~$A_{\infty}$]{KW} again, $\cC(q_{\cC}, \mathfrak{sl}_2)^{\tau} \cong \cC(-q_{\cC}, \mathfrak{sl}_2)$ as tensor categories.

To determine $q_{\cC}$, we use the intrinsic dimension of $X$: $$d(X) = e_X \circ i_X \in \mathbb{C},$$ where $e_X$ is the evaluation and $i_X$ is the coevaluation such that the rigidity zigzag morphism equals the identity. It is an invariant of the tensor category structure on $\cC$. So because the intrinsic dimension of $L_{q_{\cC}}(1)$ in $\cC(q_{\cC}, \mathfrak{sl}_2)$ is $-q_{\cC} - q_{\cC}^{-1}$ (see \cite[Exercise 8.18.8]{EGNO}), by comparing with $d(X)$ in $\cC$ we can solve for $q_{\cC}$. 

Now we apply this spirit to the category $\cC^s$: it is a rigid semisimple tensor category generated by the simple object $L[2]$ with the same fusion rule as Rep $\mathfrak{sl}_2$ (Theorem \ref{thm:main-fusion-ring}), by \cite[Theorem~$A_{\infty}$]{KW}, $\cC^s \cong \cC(q_{\cC}, \mathfrak{sl}_2)$ or $\cC(-q_{\cC}, \mathfrak{sl}_2)$ as tensor categories under a functor that sends $L[2]$ to $L_{q_\cC}(1)$ or $L_{-q_\cC}(1)$. For simplicity, we use $q_\cC$ to denote both $\pm q_{\cC}$, and we will determine $q_{\cC}$ in the rest of this section. For this, we need the intrinsic dimension of $L[2]$:
\begin{proposition}\label{dim}
  The intrinsic dimension of $L[2]$ in the category $\cC^s$ is
  $$d(L[2]) = -2.$$
\end{proposition}
\begin{proof}
The proof is similar to the proof of \cite[Theorem~4.1]{MY1}. From the fusion product decomposition $$L[2]\boxtimes L[2] \cong L[1] \oplus L[3],$$ we have for $s=1,3$, the maps $i_s:L[s]\rightarrow L[2]\boxtimes L[2]$ and $p_s:L[2]\boxtimes L[2]\rightarrow L[s]$ such that
\begin{equation}\label{eqn:c25_coev_ev}
 p_s\circ i_{s'}=\delta_{s,s'}{\rm Id}_{L[s]}\quad\text{and}\quad i_1\circ p_1+i_3\circ p_3={\rm Id}_{L[2]\boxtimes L[2]}.
\end{equation}
We take $i_1$ and $p_1$ as preliminary candidates for coevaluation and evaluation, respectively. We will calculate the rigidity zigzag morphism
\begin{align*}
 L[2]\xrightarrow{r^{-1}}L[2]\boxtimes L[1]&\xrightarrow{{\rm Id}\boxtimes i_1}L[2]\boxtimes(L[2]\boxtimes L[2])\\
 &\xrightarrow{\mathcal{A}}(L[2]\boxtimes L[2])\boxtimes L[2]\xrightarrow{p_1\boxtimes{\rm Id}} L[1]\boxtimes L[2]\xrightarrow{l}L[2],
\end{align*}
where $l$, $r$ and $\mathcal{A}$ are the left unit, right unit and the associativity isomorphisms, respectively.
Since $L[2]$ is simple and rigid, the morphism is a nonzero scalar multiple $\mathfrak{R}\cdot{\rm Id}_{L[2]}$. We need to compute the scalar $\mathfrak{R}$ and rescale $p_1$ and $i_1$ accordingly such that $\mathfrak{R} = 1$.

For $s=1,3$, we define intertwining operators $\mathcal{Y}^s_{22}:=p_s\circ\mathcal{Y}_\boxtimes$ of type $\binom{L[s]}{L[2]\,L[2]}$ and 
\begin{equation*}
 \mathcal{Y}^2_{2s}:= l\circ(p_1\boxtimes{\rm Id})\circ\mathcal{A}\circ({\rm Id}\boxtimes i_s)\circ\mathcal{Y}_\boxtimes=[l\circ(p_1\boxtimes{\rm Id})\circ\mathcal{A}\circ\mathcal{Y}_\boxtimes]\circ({\rm Id}\otimes i_s)
\end{equation*}
of type $\binom{L[2]}{L[2]\,L[s]}$. In particular,
$$\mathcal{Y}^2_{21}=\mathfrak{R}\cdot(r\circ\mathcal{Y}_\boxtimes)=\mathfrak{R}\cdot\Omega(Y_{L[2]}),$$
where $\Omega(Y_{L[2]})$ is the intertwining operator of type $\binom{L[2]}{L[2]\,L[1]}$ related to the vertex operator $Y_{L[2]}$ by skew-symmetry. By \eqref{eqn:c25_coev_ev} and the definitions of the unit and associativity isomorphisms in $\mathcal{C}^s$, we have
\begin{align}
 \mathfrak{R} \cdot\langle v, &\,\Omega(Y_{L[2]})(v,1)\mathcal{Y}^1_{22}(v,x)v\rangle +\langle v,\mathcal{Y}^2_{23}(v,1)\mathcal{Y}^3_{22}(v,x)v\rangle\nonumber\\
 & =\langle v,\mathcal{Y}^2_{21}(v,1)\mathcal{Y}^1_{22}(v,x)v\rangle+\langle v,\mathcal{Y}^2_{23}(v,1)\mathcal{Y}^3_{22}(v,x)v\rangle\nonumber\\
 & =\langle v,[l\circ(p_1\boxtimes{\rm Id})\circ\mathcal{A}\circ\mathcal{Y}_\boxtimes](v,1)\mathcal{Y}_\boxtimes(v,x)v\rangle\nonumber\\
 & =\langle v,[l\circ(p_1\boxtimes{\rm Id})\circ\mathcal{Y}_\boxtimes](\mathcal{Y}_\boxtimes(v,1-x)v,x)v\rangle\nonumber\\
 \label{eqn}& =\langle v, Y_{L[2]}(\mathcal{Y}_{22}^1(v,1-x)v,x)v\rangle,
\end{align}
where $v$ is a nonzero highest weight vector in $L[2]$ and we take $x$ to be any real number in the interval $(\frac{1}{2},1)$. We can rescale $p_1$ (and correspondingly $i_1$) if necessary to ensure the following:
\begin{align}
\langle v, \Omega(Y_{L[2]})(v,1)\mathcal{Y}^1_{22}(v,x)v\rangle &\in x^{-2h[1,2]}\big( 1+ x\,\C[[x]]\big),\label{phi2}\\
\langle v,\mathcal{Y}^2_{23}(v,1)\mathcal{Y}^3_{22}(v,x)v\rangle & \in x^{h[1,3]-2h[1,2]}\C[[x]],\label{phi1}\\
\langle v, Y_{L[2]}(\mathcal{Y}_{22}^1(v,1-x)v,x)v\rangle &\in (1-x)^{-2h[1,2]}\left(1+\frac{1-x}{x}\C\left[\left[\frac{1-x}{x}\right]\right]\right).\label{psi}
\end{align}
Recall that $h[p,r]$ is the lowest conformal weight of $L[p,r]$. In particular, $h[1,2]=-\frac{1}{2}$ and $h[1,3]=-1$.

Now using the fact that the singular vector $W(-1)v$ and the subsingular vector $$\left(L(-1)^2+\frac{6}{c}W(-2)\right)v$$ are equal to $0$ in $L[2]$, the above products and iterates of intertwining operators are solutions to the differential equation $\phi''(x)=0$ (see for example \cite[Sec.~4.1]{CMY2} for the details). We take a fundamental set of solutions to be 
\[
\phi_1(x) = 1, \qquad \phi_2(x) = x.
\]
Then it is clear that \eqref{phi2} equals $\phi_2(x)$ and \eqref{phi1} is some linear combination $a\,\phi_1(x)+b\,\phi_2(x)$, so by \eqref{eqn},
\begin{equation*}
 \langle v, Y_{L[2]}(\mathcal{Y}_{22}^1(v,1-x)v,x)v\rangle =a\,\phi_1(x)+(\mathfrak{R}+b)\phi_2(x)
\end{equation*}
for $x$ such that both sides converge. Then from \eqref{psi}, the iterate on the left side here is a solution $\psi(x)$ to the differential equation such that $(1-x)^{-1}\psi(x)$ is analytic near $x=1$, with constant term $1$ when expanded as a series in $\frac{1-x}{x}$. The only solution with this property is $\phi_1(x)-\phi_2(x)$, so we have $a=1$ and $\mathfrak{R}+b=-1$.

Now we need to compute $b$ in order to obtain $\mathfrak{R}$. We first note that from the $L_0$-conjugation formula,
\begin{align*}
 \phi_1(x)+b\,\phi_2(x) & =\left\langle v,\mathcal{Y}_{23}^2(v,1)\mathcal{Y}_{22}^3(v,x)v\right\rangle\nonumber\\
 & =\left\langle v,\mathcal{Y}^2_{23}(v,1)x^{L_0-2h[1,2]}\mathcal{Y}^3_{22}(v,1)v\right\rangle\nonumber\\
 & =\sum_{n\geq 0}\left\langle v,\mathcal{Y}^2_{23}(v,1)\pi_n(\mathcal{Y}^3_{22}(v,1)v)\right\rangle x^{n},
\end{align*}
where $\pi_n$ denote the projection of $L[3]$ to $L[3]_{(n-1)}$, the subspace of conformal weight $n-1$. Denote $$c_n=\left\langle v,\mathcal{Y}^2_{23}(v,1)\pi_n(\mathcal{Y}^3_{22}(v,1)v)\right\rangle,$$ because $\phi_1(x) = 1$ and $\phi_2(x) = x$, we have $c_0=1$ and $c_1=b$. We will proceed to calculate $c_1$ using $c_0$.

Next we compute $\pi_1(\mathcal{Y}_{22}^3(v,1)v)$ in terms of the lowest weight vector $v_{3}=\pi_0(\mathcal{Y}_{22}^3(v,1)v)$ of $L[3]$. Let $\langle\cdot,\cdot\rangle$ denote the unique invariant bilinear form on $L[3]$ such that $\langle v_3,v_3\rangle =1$; it is nondegenerate on the subspace $L[3]_{(0)}$ which is 1-dimensional with basis $\lbrace L_{-1}v_3\rbrace$. 
Now using the commutator formula,
\begin{align*}
 \left\langle L_{-1} v_3, \mathcal{Y}^3_{22}(v,1)v\right\rangle & =\left\langle v_3, L_1\mathcal{Y}^3_{22}(v,1)v\right\rangle\nonumber\\
 & =\left\langle v_3,\mathcal{Y}^3_{22}((L_{-1}+2L_0)v,1)v\right\rangle\nonumber\\
 & =\left\langle v_3,[L_0,\mathcal{Y}^3_{22}(v,1)]v\right\rangle+h[1,2]\langle v_3,v_3\rangle\nonumber\\
 & =h[1,3] = -1.
\end{align*}
Also, because $\langle L_{-1} v_3,L_{-1}v_3\rangle = -2$,
it is easy to show
\begin{align*}
 \pi_1(\mathcal{Y}^3_{22}(v,1)v) & =\frac{1}{2}L_{-1}v_3.
\end{align*}
Now we can compute
\begin{align*}
 c_1 & =\left\langle v,\mathcal{Y}^2_{23}(v,1)\pi_1(\mathcal{Y}^3_{22}(v,1)v)\right\rangle\nonumber\\
 & =\frac{1}{2}\left\langle v,\mathcal{Y}^2_{23}(v,1)L_{-1}v_3\right\rangle =-\frac{1}{2}\left\langle v,\mathcal{Y}^2_{23}(L_{-1}v,1)v_3 \right\rangle\nonumber\\
 & =-\frac{1}{2}\left\langle v, \left[(\mathrm{ad}\,L_0-h[1,2])\mathcal{Y}^2_{23}(v,1)\right]v_3 \right\rangle\nonumber\\
 & =\frac{1}{2}h[1,3]\left\langle v,\mathcal{Y}^2_{23}(v,1)v_3\right\rangle =-\frac{1}{2}c_0 = -\frac{1}{2}.
\end{align*}
Thus $b = c_1 = -\frac{1}{2}$ and the rigidity scalar $\mathfrak{R}$ is
\begin{equation*}
 \mathfrak{R}=-1-b=-\frac{1}{2}.
\end{equation*}
Therefore, we have to rescale the evaluation to be $-2\,p_1$ and keep the coevaluation $i_1$. The intrinsic dimension would equal $-2$.
\end{proof}

\begin{remark}
    Proposition \ref{dim} also shows that the generating object $L[2]$ is rigid, then using the fusion rules, we also obtain that all the simple objects are rigid.
\end{remark}

Now we go back to the tensor equivalence $\cC^s \cong \cC(q_{\cC}, \mathfrak{sl}_2)$ for some $q_\cC \in \C$ under the functor $L[2] \mapsto L_{q_\cC}(1)$. Because intrinsic dimension is a tensor invariant, 
\[
d(L[2]) = -2 = d(L_{q_\cC}(1)) = -q_\cC - q_{\cC}^{-1}.
\]
It follows that $q_{\cC} = 1$. Therefore, we prove
\begin{theorem}
    There are tensor equivalences $$\cC^s \cong \cC(1, \mathfrak{sl}_2) \cong {\rm Rep}\; \mathfrak{sl}_2.$$
\end{theorem}

The category $\cC^s$ is also related to the Virasoro tensor categories via its equivalence to Rep $\mathfrak{sl}_2$. In \cite{MY2}, McRae and the fourth named author studied the category $\mathcal{O}_p$ of finite length modules for the Virasoro algebra at central charge $c_{1,p} = 13 - 6p -\frac{6}{p}$ whose simple objects are labeled by $\mathcal{L}_{r,s}$ for $r\in \N$ and $1 \leq s \leq p$. Let $\mathcal{O}_p^L$ denote the subcategory consisting of direct sums of the simple objects $\mathcal{L}_{r,1}$ for $r \in \N$. Then the proof of \cite[Theorem~4.3]{MY2} says $\mathcal{O}_p^L$ is equivalent to Rep $\mathfrak{sl}_2$ if $p$ is even; and equivalent to $({\rm Rep}\; \mathfrak{sl}_2)^\tau$ if $p$ is odd. It follow that
\begin{corollary}
    The category $\cC^s$ is equivalent to the Virasoro tensor category $\mathcal{O}_p^L$ for even $p$.
\end{corollary}

\bigskip

\subsection*{Acknowledgement}    
 D. Adamovi\'c  and G. Radobolja were  partially supported by  Croatian Science Foundation under the project IP-2022-10-9006  and by the project “Implementation of cutting-edge research and its application as part of the Scientific Center of
Excellence for Quantum and Complex Systems, and Representations of Lie Algebras", Grant No. PK.1.1.10.0004, co-financed by the European Union through the
European Regional Development Fund - Competitiveness and Cohesion Programme 2021-2027.

G. Radobolja was partially supported by the European Union (Next\-GenerationEU) under the Croatian Recovery and Resilience Plan 2021\-–\-2026 (NRRP), through the University of Split institutional project “Mathematical modelling and simulations of physical systems (MaMoS) IP-UNIST-46”, approved by the Ministry of Science, Education and Youth of the Republic of Croatia.

J. Yang and M. Peng was supported by China NSF grant 12371030 and a startup grant from Shanghai Jiao Tong University.


\begin{thebibliography}{}

\bibitem{AR1} D. Adamovi\'{c}, G. Radobolja, Free field realization of the twisted Heisenberg--Virasoro algebra at level zero and its applications, {\em J. Pure Appl. Alg.} {\bf 219} (2015), no. 10, 4322--4342; arXiv:1405.1707.

\bibitem{AR2} D. Adamovi\' c, G. Radobolja, On free field realization of $W(2,2)$-modules, {\em SIGMA} {\bf 12} (2016), 113, 13 pp; arXiv:1605.08608.

\bibitem{AR3} D. Adamovi\'{c}, G. Radobolja, Self-dual and logarithmic representations of the twisted Heisenberg-Virasoro algebra at level zero, {\em Comm. Contemp. Math.} {\bf 21} (2019), no. 02, 1850008; arXiv:1703.00531 [math.QA].

\bibitem{ALY} D. Adamovi\'{c}, X. Lin, J. Yang, Tensor category of $\Z_2$-orbifold of Heisenberg vertex operator algebra and its applications, arXiv: 2604.12120.


\bibitem{Billig} Y. Billig, Representations of the twisted Heisenberg-Virasoro algebra at level zero, {\em Canadian Math. Bulletin}, {\bf 46}
(2003), 529--537, arXiv:0201314v1.

\bibitem{CMY2}
T. Creutzig,  R.  McRae, J. Yang, On ribbon categories for singlet vertex algebras, {\em Comm. Math.
Phys.} {\bf 387} (2021), 865–925.

\bibitem{C.et.al}
T. Creutzig, C. Jiang, F. Orosz Huniker, D. Ridout, J. Yang,
Tensor categories arising from the Virasoro algebra,
{\em Adv.\ Math.} {\bf 380} (2021), 107621.

\bibitem{EGNO} P. Etingof, S. Gelaki, D. Nikshych and V. Ostrik, \textit{Tensor Categories},  Mathematical Surveys and Monographs, \textbf{205}, American Mathematical Society, Providence, RI, 2015, xvi+343 pp.

\bibitem{EP} 
P. Etingof and  D. Penneys,  Rigidity of non-negligible objects of moderate growth in braided categories, {\em Forum Math., Pi} {\bf 14} (2026), Paper No. e7, 18 pp.

\bibitem{FZ}
I.~B.~Frenkel and Y.-C.~Zhu, Vertex operator algebras associated to representations of affine
and Virasoro algebras, \textit{Duke Math. J.} \textbf{66} (1992), no.~1, 123--168.

\bibitem{H}
Y.-Z. Huang, $C_1$-cofiniteness and vertex tensor categories, arXiv:2509.20737.

\bibitem{JLPZ} 
W. Jiang, D. Liu, Y. Pei, and K. Zhao, Singular vectors, characters, and composition series for the $N=1$ BMS superalgebra, arxiv:2412.17000.

\bibitem{JZ} 
W. Jiang, and W. Zhang, Verma modules over the $W(2,2)$ algebras, {\em J. Geom. Phys.} {\bf 98} (2015), 118--127.

\bibitem{KW}
D. Kazhdan and H. Wenzl, Reconstructing monoidal categories, {\em I. M. Gel’fand Seminar}, 111–136,
Adv. Soviet Math., {\bf 16}, Part 2, Amer. Math. Soc., Providence, RI, 1993.

\bibitem[L]{L}
G. Lusztig, {\em Introduction to Quantum Groups}, Progress in Mathematics, \textbf{110}, Birkh\"{a}user Boston, Inc., Boston, MA, 1993, xii+341 pp.


\bibitem{Li} 
H.~Li, Determining fusion rules by $A(V)$-modules and bimodules,
\textit{J. Alg.} \textbf{212} (1999), no.~2, 515--556.

\bibitem{Miy}  
M. Miyamoto, $C_1$-cofiniteness and fusion products of vertex operator algebras, {\em Math. Lect. Peking
Univ.} Springer, Heidelberg, 271–279 (2014); arXiv:1305.3008.

\bibitem{MY1}
R. McRae and J. Yang, An $\mathfrak{sl}_2$-type tensor category for the Virasoro algebra at central charge 25 and applications, {\em Math. Z.}, {\bf 303} (2023), no. 2, Paper No. 32, 40 pp.

\bibitem{MY2}
R. McRae and J. Yang, Structure of Virasoro tensor categories at central charge $13-6p-6p^{-1}$ for integers $p>1$, {\em Trans. Amer. Math. Soc.} {\bf 78} (2025), no. 10, 7451-7509; arXiv:2011.02170.

\bibitem{R} 
G. Radobolja, Subsingular vectors in Verma modules, and tensor product of weight modules over the twisted Heisenberg--Virasoro algebra and $W(2,2)$ algebra, {\em J. Math. Phys.} {\bf 54} (2013), 071701; arXiv:1302.0801

\bibitem{Zhang-Dong} 
W. Zhang and C. Dong, $W$-algebra $W(2,2) $ and the vertex operator algebra $L(\frac{1}{2},0) \otimes L( \frac{1}{2},0)$, {\em Comm. Math. Phys}. {\bf 285} (2009), 991--1004.
\end{thebibliography}
\end{document}